\documentclass[11pt,leqno]{article}

\usepackage[utf8]{inputenc}
\usepackage{t1enc}
\usepackage{subfiles}
\usepackage{enumitem}
	\setlist[enumerate]{label*=\arabic*.}

\usepackage{amsfonts}
\usepackage{amsthm}
\usepackage{amsmath}
\usepackage{amssymb}
\usepackage{amscd}
\usepackage{mathtools}
\usepackage{esint}
\usepackage[mathscr]{eucal}
\usepackage{accents}

\usepackage[margin=3cm]{geometry}
\usepackage{indentfirst}
\usepackage{graphicx}
\usepackage{graphics}
\usepackage{lscape}
\usepackage{tikz-cd}
\usepackage{color}
\usepackage{pict2e}
\usepackage{epic}
\usepackage{epstopdf}
\usepackage{titlesec}
	\titleformat{\section}[block]{\Large\bfseries\filcenter}{\thesection}{1em}{}
\usepackage{commath}
\usepackage{float}
\usepackage{caption}
\usepackage{etoolbox}
\usepackage[affil-it]{authblk}
\usepackage{combelow}

\usepackage[hidelinks]{hyperref}
\hypersetup{bookmarksopen=true} 
\usepackage{hypcap}

\graphicspath{{./Pictures/}}

\swapnumbers
\numberwithin{equation}{section}

\theoremstyle{plain}
\newtheorem{teo}[equation]{Theorem}
\newtheorem*{teo*}{Theorem}
\newtheorem{lema}[equation]{Lemma}
\newtheorem{prop}[equation]{Proposition}

\theoremstyle{definition}

\newtheorem{remark}[equation]{Remark}

\newcommand{\thistheoremname}{}
\newtheorem{genericthm}[equation]{\thistheoremname}

\newcommand{\thistheoremnames}{}
\newtheorem*{genericthms}{\thistheoremnames}
\newenvironment{para*}[1]
  {\renewcommand{\thistheoremnames}{#1}%
   \begin{genericthms}}
  {\end{genericthms}}

\expandafter\let\expandafter\oldproof\csname\string\proof\endcsname
\let\oldendproof\endproof
\renewenvironment{proof}[1][\proofname]{%
  \oldproof[\upshape \bfseries #1:]%
}{\oldendproof}

\graphicspath{{./Pictures/}}

\def \a{\alpha}
\def \R {\mathbb{R}}

\def \C{\mathbb{C}}

\def \N{\mathbb{N}}
\def \e{\varepsilon}
\def \p{\partial}
\def \d{\,\textup{d}}

\def \mc{\mathcal}

\def \eu{\EuScript}

\def \diag{\textup{diag}}

\def \hs{\hspace{0.5cm}}
\def \tp{\textup}

\newcommand*{\ou}[2]{\overset{\text{\large ${#1}$}}{#2}}

\def\Xint#1{\mathchoice
{\XXint\displaystyle\textstyle{#1}}%
{\XXint\textstyle\scriptstyle{#1}}%
{\XXint\scriptstyle\scriptscriptstyle{#1}}%
{\XXint\scriptscriptstyle\scriptscriptstyle{#1}}%
\!\int}
\def\XXint#1#2#3{{\setbox0=\hbox{$#1{#2#3}{\int}$ }
\vcenter{\hbox{$#2#3$ }}\kern-.6\wd0}}

\def\dashint{\Xint-}

\begin{document}

\title{\LARGE \textbf{Extremal rank-one convex integrands\\ and a conjecture of \v Sverák}}

\author{\Large André Guerra } 

\affil{\small University of Oxford, Andrew Wiles Building
 Woodstock Rd,
Oxford OX2 6GG,
United Kingdom  \protect\\
  {\tt{guerra@maths.ox.ac.uk}}}

\date{}

\maketitle

\begin{abstract} 
We show that, in order to decide whether a given probability measure is laminate, it is enough to verify Jensen's inequality in the class of extremal non-negative rank-one convex integrands.
We  also identify a subclass of these extremal  integrands, consisting of truncated minors, thus proving a conjecture made by \textsc{\v Sverák} in (\textit{Arch. Ration. Mech. Anal. 119 293-300, 1992}).
\end{abstract}


\section{Introduction}

Since its introduction in the seminal work  of \textsc{Morrey} \cite{Morrey1952},
quasiconvexity has played an important role not just in
the Calculus of Variations \cite{Chen2017, Evans1986, Iwaniec1994, Rindler2018} but also in problems from other areas of
mathematical analysis, for instance in the theory of compensated
compactness  \cite{Murat1978, Tartar1979}.
Nonetheless, this concept is still poorly
understood and has been mostly studied in relation with
polyconvexity and rank-one convexity, which are respectively stronger
and weaker notions that are easier to deal with (we refer the reader
to Section \ref{sec:prelims} for terminology and notation).
An outstanding open problem in the area is Morrey's problem,
which is the problem of deciding whether rank-one convexity implies
quasiconvexity, so that the two notions coincide. 
A fundamental example \cite{Sverak1992a} of \textsc{\v Sverák} shows that this implication does not hold in dimensions $3\times
2$ or higher and, more recently, \textsc{Grabovsky}
\cite{Grabovsky2018} has found a different
example in dimensions $8\times 2$ which moreover is $2$-homogeneous. 
The problem in dimensions $2\times 2$, in
particular, remains completely open,
but in the last two decades evidence 
towards a positive answer in this case has been piling up \cite{Astala2012, Faraco2008,
  Harris2018,Kirchheim2016, Kirchheim2008, Muller1999a,
  Pedregal1996,Pedregal1998,Sebestyen2017}.

A presumably easier (but still unsolved) version of Morrey's problem in dimensions $2\times 2$ is to
decide whether rank-one convex integrands\footnote{We will refer to real-valued functions defined on a matrix space as \emph{integrands}; this terminology is standard in the Calculus of Variations literature.} in the space of $2\times 2$
symmetric matrices are quasiconvex. In this direction,
\textsc{\v Sverák} introduced in \cite{Sverak1992} new quasiconvex
integrands, which were later generalized in \cite{Faraco2003}. 
For any $n\times n$ symmetric matrix $A$, these integrands are
defined by
$$
F_k(A)\equiv \begin{cases}
|\det A| & A\textup{ has index } k\\
0 & \textup{otherwise}
\end{cases}
$$
for $k=0,\dots, n$; we recall that the \emph{index} of a matrix is the number of its negative eigenvalues. We also note that the integrand $F_0$ is sometimes called $\det^{++}$ in the literature, since its support is the set of positive definite matrices. These integrands have played an important role in studying other problems related to the Calculus of Variations, for instance in building counterexamples to the regularity of elliptic systems \cite{Muller2003} or in the computation of rank-one convex hulls of compact sets \cite{Szekelyhidi2005}. 

In order to understand \textsc{\v Sverák}'s motivation for considering these integrands it is worth making a small excursion into classical convex analysis. Given a real vector space $\mathbb{V}$ and a convex set $K\subset \mathbb{V}$, one can define the set of \emph{extreme points} of $K$ as the set of points which are not contained in any open line segment contained in $K$.  In general the set of extreme points might be very small: this is what happens, for instance, when the set is a convex cone $C\subset \mathbb{V}$, since in this situation all non-zero vectors  are contained in a ray through zero. However, if we can find a convex base $B$ for $C$, then we note that such a ray corresponds  to a unique point in $B$.  If this is an extreme point of $B$ then we say that we have an \emph{extremal ray}.

We are interested in the extremal rays of the cone $C$ of rank-one convex integrands. This cone has the inconvenient feature that it is not \emph{line-free}: there is a set of elements $v\in \mathbb{V}$ such that, for any $c\in C$ and any $t\in \R$, the point $c+tv$ is in $C$; this set is precisely $C\cap(-C)$. In turn, it is quite clear that this is the set of rank-one affine integrands. A reasonable way of disposing of rank-one affine integrands is by demanding non-negativity from  all integrands from $C$. This leads us to the definition of extremality considered by 
\textsc{\v Sverák}: we say that a non-negative rank-one convex integrand $F$ is \emph{extremal}	
if, whenever we have $F=E_1+E_2$ for $E_1, E_2$ non-negative rank-one convex integrands, 
then each $E_i$ is a non-negative multiple of $F$.
A weaker notion of extremality was introduced by \textsc{Milton} in
\cite{Milton1990}
 for the case of quadratic forms (see also \cite{Harutyunyan2017}) but
 we shall not  discuss it further here.

Let us now  explain the relation between  \textsc{\v Sverák}'s integrands and extremality. In \cite{Sverak1992} \textsc{\v Sverák} observed that the 
polyconvex integrand 
$|\det|\colon \R^{n\times n} \to [0,\infty)$ is not extremal, since
$|\det|=\det^++\det^-$, where as usual
$$\textup{det}^+\equiv \max(0,\det), \hs
\textup{det}^-\equiv \max(0,-\det),
$$
which are also polyconvex.
He  also observed that $\det^\pm$ are not extremal in $\R^{n\times n}_{\textup{sym}}$, since
\begin{equation*}
\textup{det}^+= \sum_{k \textup{ is even}} F_k \textup{ and } 
\textup{det}^-= \sum_{k \textup{ is odd}} F_k\hs
 \textup{
  in } \R^{n\times n}_{\textup{sym}}.\label{eq:3}
\end{equation*}
However, he conjectured that $\det^\pm$ are extremal in $\R^{n\times n}$ and also that each $F_k$ is extremal in $\R^{n\times n}_\textup{sym}$.
In this paper we  give an affirmative answer to both
conjectures\footnote{In fact, \textsc{\v Sverák} only conjectures
  extremality in the cone of quasiconvex integrands, so our results are
  in this sense slightly stronger than his conjecture.}:

\begin{teo}\label{teo:intro}
Given a minor $M\colon \R^{n\times n}\to \R$, let $M^\pm$ be its positive and negative parts. Then $M^\pm$ are extremal non-negative rank-one convex integrands in $\R^{n\times n}$.

For  $k=0,\dots,n$, the integrands
 $F_k\colon \R^{n\times n}_{\textup{sym}}\to \R$, are extremal non-negative
 rank-one convex integrands in $\R^{n\times n}_{\textup{sym}}$.
\end{teo}

As a main tool we use the fact that, on a connected open set, a rank-one affine integrand is an affine combination of minors; this is proved by localizing the arguments from
\textsc{Ball--Currie--Olver} \cite{Ball1981} concerning the classification of Null Lagrangians.

The importance of extreme points in convex analysis has to do with
the Krein--Milman theorem, which states that 
the closed convex hull of the set of extreme points of a compact,
convex subset $K$ of a locally convex vector space 
is the whole set $K$---informally, this means that the set of extreme points is a set
of ``minimal information'' needed to recover $K$.
However, \textsc{Klee} \cite{Klee1959} showed that Krein--Milman theorem is generically true for trivial reasons: if we fix an infinite-dimensional Banach space and we consider the space of its compact, convex subsets (which we can equip with the Hausdorff distance so that it becomes a complete metric space), then 
for almost every compact convex set $K$ its extreme points are dense in $K$. Here we mean ``almost every'' in the sense that the previous statement is false only in a meagre set. Despite this disconcerting result, the situation can be somewhat remedied with the help of Choquet theory, which roughly states that, under reasonable assumptions, an arbitrary point in $K$ can be represented by a measure carried in the set of its extreme points.
For precise statements and much more information concerning Choquet theory we refer the reader to the lecture notes \cite{Phelps2001} or to the monograph \cite{Lukes2010} and for
Krein--Milman-type theorems for semi-convexity notions see \cite{Kirchheim2001, Kruzik2000, Matousek1998}.

Theorem \ref{teo:intro}  can be interpreted in light of Choquet theory.
It is well-known (see, for instance, the lecture notes \cite{Muller1999a}) that Morrey's problem is equivalent to a dual problem, which is the problem of deciding whether the class of homogeneous gradient Young measures and the class of laminates are the same.
Fix a compactly supported  Radon probability measure $\nu$ on $\R^{n\times n}$; for simplicity we assume that $\nu$ has support contained in the interior of the cube $Q\equiv \prod_{i=1}^{n^2} [0,1]\subset \R^{n^2}\cong \R^{n\times n}.$ In order to decide whether $\nu$ is a laminate, we can resort to an important theorem due to \textsc{Pedregral} \cite{Pedregal1993}, which states that $\nu$ is a laminate if and only if Jensen's inequality holds for any rank-one convex integrand $f\colon \R^{n\times n}\to \R$:
$$ f(\overline \nu)\leq \langle \nu , f\rangle , \hs \overline \nu \equiv \int_{\R^{n\times n}} A\textup{ d} \nu(A).$$
Since the class of rank-one convex integrands is rather large, the problem of deciding whether a measure is a laminate is in general very hard. However,
it follows from Choquet theory that one does not need to test $\nu$ against all rank-one convex integrands, it being sufficient to test it in a strictly smaller class:

\begin{teo}\label{teo:choquet}
Let $A_1,\dots, A_{2^{n^2}}$ be the vertices of the cube $Q$. A Radon probability measure $\nu$ supported on the interior of $Q$ is a laminate if and only if
$$g(\overline \nu)\leq \langle \nu, g \rangle $$
for all integrands $g$ which are extreme points of the convex set
$$\left\{f\colon Q\to [0,\infty):  \sum_{i=1}^{2^{n\times n}} f(A_i)= 1, f \textup{ is rank-one convex} \right\}.$$
\end{teo}

We note that the summation condition is simply a normalization which corresponds to fixing a base of the cone of non-negative rank-one convex integrands on $Q$.

Our interest in extremal integrands was ignited by the work \cite{Astala2015} of 
\textsc{Astala--Iwaniec--Prause--Saksman}, 
where it was shown 
that an integrand known as Burkholder's function is
extreme in the class of  homogeneous, isotropic, rank-one convex
integrands; 
in fact, this integrand is also
the least integrand  in this class, in the sense that no other element of the class is below it at all points.
The relevance of this fact is readily seen: from standard results about quasiconvex envelopes it follows immediately  that the Burkholder function is either quasiconvex everywhere or quasiconvex nowhere.
Burkholder's function was found in the context of martingale theory
by \textsc{Burkholder} \cite{Burkholder1984, Burkholder1988}   and was later generalized to higher dimensions
by \textsc{Iwaniec} \cite{Iwaniec2002}.
This remarkable function is a bridge between Morrey's problem and
important problems in Geometric Function Theory \cite{Astala2009, Iwaniec2001} and
we refer the reader to the very interesting papers \cite{Astala2012,
Astala2015, Iwaniec2002} and the references therein for more details
in this direction.


Finally, we give a brief outline of the paper.
Section \ref{sec:prelims} contains standard definitions, notation and briefly recalls some useful facts for the reader's convenience. Section \ref{sec:hom} comprises results concerning improved homogeneity properties of rank-one convex integrands which vanish on isotropic cones. Section \ref{sec:maintheorem} is devoted to the proof of Theorem \ref{teo:intro}. Finally, Section \ref{sec:choquet} elaborates on the relation between Choquet theory and Morrey's problem and we prove  Theorem \ref{teo:choquet}.

\section{Preliminaries}\label{sec:prelims}

In this section we will gather a few definitions and notation for the
reader's convenience. The material is standard and can be found for
instance in the excellent references \cite{Dacorogna2007} and  \cite{Muller1999a}.

Consider an integrand $E\colon\R^{m\times n}\to \R$. We say that $E$ is \emph{polyconvex} 
if $E(A)$ is a convex function of the minors of $A$,  see \cite{Ball1977}.
If $E$ is locally integrable, we say that $E$ is \emph{quasiconvex}
if there is some bounded open set $\Omega\subset \R^n$ such that, for
all $A\in \R^{m\times n}$ and all $\varphi \in C^{\infty}_0(\Omega,\R^n)$,
$$E(A)\leq \dashint_\Omega E(A+ D \varphi(x)) \d x.$$
This notion was introduced in \cite{Morrey1952} and generalized to
higher-order derivatives by \textsc{Meyers} in \cite{Meyers1965}.
The case of higher derivatives was also addressed in \cite{Ball1981}, where it was shown that quasiconvexity is not implied by the corresponding notion of rank-one convexity if $m>2$ and $n\geq 2$.
The integrand $E$ is \emph{rank-one convex} if, for
all matrices $A, X\in \R^{m\times n}$ such that $X$ has rank one, the function
$$t\mapsto E(A+ t X) $$
is convex. An equivalent definition of rank-one convexity
is that $E$ is rank-one convex if, whenever $(A_i, \lambda_i)_{i=1}^N$ satisfy the $(H_N)$
conditions (c.f. \cite[Def. 5.14]{Dacorogna2007}), we have
$$\sum_{i=1}^NE(\lambda_i A_i)\leq \sum_{i=1}^N \lambda_i E(A_i).$$
By a slight abuse of terminology we will call a collection $(A_i,
\lambda_i)_{i=1}^N$  which satisfies the $(H_N)$ conditions a \emph{prelaminate}.
It is also clear how to adapt the definition of rank-one convexity to the more general situation where $E$ is  defined on an open set $\mc{O}\subset \R^{m\times n}$.
Finally, $E\colon \R^d\to \R$ is \emph{separately convex} if, for all $x\in \R^d$ and all $i=1,\dots, d$, the function $t\mapsto E(x+t e_i)$ is convex; we denote by $e_1,\dots, e_d$  the standard basis of $\R^d$.

We recall that all real-valued rank-one convex integrands are locally Lipschitz continuous, a fact
that we will often use implicitly. Moreover,
if $(m,n)\neq (1,1)$, we have
$$E \textup{ convex }\raisebox{-5pt}{\,$\ou{\Rightarrow}{\not
    \Leftarrow}$\,}
E\textup{ polyconvex }\raisebox{-5pt}{\,$\ou{\Rightarrow}{\not
    \Leftarrow}$\,}   E \textup{ quasiconvex } \Rightarrow E \textup{ rank-one convex }$$
while $E$ quasiconvex $\Leftarrow E$ rank-one convex fails if $n\geq
2, m\geq 3$. The case $n\geq m=2$ is the content of Morrey's problem.

We will only consider square matrices, so $n=m$. 
Besides polyconvexity, the integrands $\det^\pm\colon \R^{n\times n} \to [0,\infty)$ have two important
properties: they are positively $n$-homogeneous (in fact, they are
$n$-homogeneous if $n$ is even) and isotropic. A generic integrand
$E\colon \R^{n\times n}\to \R$ is said to be
\emph{$p$-homogeneous}
for a number $p\geq 1$ if 
$$E(tA)=|t|^p E(A)\hs  \textup{ for all } A\in \R^{n\times n}\textup{
  and all } t\in \R;$$
it is \emph{positively $p$-homogeneous} if the same holds only for $t>0$.
The integrand $E$ is \emph{isotropic} if it is invariant
under the left-- and right--$\textup{SO}(n)$ actions, that is, 
$$E(QAR)= E(A) \hs \textup{for all } Q,R\in \textup{SO}(n).$$
This condition can be understood in a somewhat more concrete way with
the help of singular values. The \emph{singular values} $\widetilde \sigma_1(A)\geq \dots\geq
\widetilde\sigma_n(A)\geq0$ of a matrix $A$ are the eigenvalues of the matrix
$\sqrt{AA^T}$. We shall consider the \emph{signed} singular values
$\sigma_j(A)$, which are defined by
$$
\sigma_j(A)=\widetilde\sigma_j(A) \textup{ for } 1\leq j\leq n-1, \hs
\sigma_n(A)=\textup{sign}(\det A) \widetilde\sigma_n(A).$$
As we shall only deal with the signed singular values, the word signed
will sometimes be omitted.
The importance of singular values is largely due to the
following standard result (c.f. \cite[\S 13]{Dacorogna2007}):
\begin{teo}\label{teo:singval}
  Let $A\in \R^{n\times n}$. There are matrices $Q,R\in
  \textup{SO}(n)$ and real numbers $|\sigma_n|\leq \sigma_{n-1}\leq \dots
  \leq \sigma_1$ such that
$$A=R\Sigma Q, \hs \Sigma=\diag(\sigma_1,\dots,\sigma_n).$$
\end{teo}
With the help of this theorem we can reinterpret isotropy as follows:
consider the polar decomposition $A=Q\sqrt{AA^T}$ for some $Q\in
\textup{O}(n)$ and consider a diagonalization of the
positive-definite matrix $\sqrt{AA^T}$, so 
$R\sqrt{AA^T}R^{-1}=\textup{diag}(\widetilde \sigma_1,\dots,\widetilde\sigma_n)$
for some  $R\in
\textup{SO}(n)$. Here all the $\widetilde\sigma_j$'s are positive, but by
flipping the sign of the last one if $\det A<0$ we can take $Q\in \textup{SO}(n)$.
Hence, if $E$ is isotropic,
$$E(A)=E(Q\sqrt{A A^T})=E( R\textup{diag}(\sigma_1,\dots,\sigma_n) R^T) =
E(\sigma_1,\dots, \sigma_n).$$
 So isotropic integrands are functions of the signed singular values.

When $n=2$ isotropy is particularly simple to handle, since there is a simple way of understanding the signed singular values of a matrix $A\in \R^{2\times 2}$. For this, we recall the conformal--anticonformal decomposition:  we can write
$$A\equiv \begin{bmatrix}
a & b\\ c & d
\end{bmatrix}
= \frac 1 2 \begin{bmatrix}
a+d &b-c\\
c-b & a+d
\end{bmatrix}
+\frac 1 2
\begin{bmatrix}
a-d &b+c\\
b+c & d-a
\end{bmatrix}
\equiv A^++A^-.
$$
This corresponds to the decomposition
$$\R^{2\times 2} = \R^{2\times 2}_{\textup{conf}}\oplus \R^{2\times
  2}_{\textup{aconf}},$$which is orthogonal with respect to the Euclidean inner
product. Here $\R^{2\times 2}_{\textup{conf}}$ is the space of
\emph{conformal matrices} while $\R^{2\times 2}_{\textup{aconf}}$
corresponds to the
\emph{anticonformal matrices}; these are the matrices
that are scalar multiples of orthogonal matrices and have respectively
positive and negative determinant.
This decomposition is particularly
useful for us because the singular values of $A$ satisfy the 
identities
\begin{equation*}
\begin{split}
\sigma_1(A)&= \frac 1 2 \left(\sqrt{|A|^2+ 2 |\det A|}+\sqrt{|A|^2 - 2 |\det A|}\right)=\frac{1}{\sqrt 2}(|A^+|+|A^-|),\\
\sigma_2(A)&= \frac 1 2 \left(\sqrt{|A|^2+ 2 |\det A|}-\sqrt{|A|^2 - 2 |\det A|}\right)=\frac{1}{\sqrt2} (|A^+|-|A^-|),
\end{split}
\end{equation*}
where $|A |^2=\textup{tr}(A^T A)$ denotes the usual Euclidean norm of
a matrix. Hence, if $n=2$ and $E$ is isotropic, $E(A)=E(|A^+|,|A^-|)$.
In particular, the above formulae yield
$$2 \det A = |A^+|^2-|A^-|^2.$$
The
above decomposition also allows one to identify $\R^{2\times 2}\cong \C^2$
by the linear isomorphism
$$
\begin{bmatrix}
  a & b\\ c&d
\end{bmatrix}
\mapsto \left((a+d) + i(c-b), (a-d)+i (b+c)\right)$$
The advantage of this identification is that we can say that
a integrand $E\colon\C^2\to \R$ is rank-one convex if and only if the function
$$t\mapsto E(z+t \xi, w+t \zeta)$$ is convex for all $(z,w)\in \C^2$
and all $(\xi, \zeta)\in S^1\times S^1$.

The conformal--anticonformal decomposition of $\R^{2\times 2}$ is also related to an important rank-one convex integrand, commonly referred to as \emph{Burkholder's function}. This function can be defined in any real or complex Hilbert space with the norm $\Vert\cdot \Vert$ by
$$B_p(x,y)=((p^*-1) \Vert x \Vert - \Vert y \Vert) ( \Vert x \Vert + \Vert y \Vert)^{p-1},$$
where $p^*-1 = \max(p-1, (p-1)^{-1})$; here and in the rest of the paper, when referring to $B_p$, we implicitly assume that $1<p<\infty$. This remarkable function is zig-zag convex, i.e.
$$t\mapsto B_p(x+t a, y + tb) \textup{ is convex  whenever } \norm{a}\geq \norm{b},$$
see \cite{Burkholder1988}, and it is also $p$-homogeneous. If the Hilbert space where $B_p$ is defined is $\C$, the zig-zag convexity of $B_p$ implies that the Burkholder function $B_p\colon \C\times \C\to \R$ is rank-one convex. Since we are interested in non-negative integrands, we will also deal with the integrand $B_p^+\equiv \max(B_p,0)$, which is also rank-one convex. Moreover, $B_2^+=\det^+$, so $B_p^+$ can be seen as a ``$\det^+$-type integrand'', in the sense that it is rank-one convex, isotropic and vanishes on some cone 
$$\eu{C}_{p^*-1}=\{(z,w)\in \C\times \C: (p^*-1)|z|=|w|\},$$ but it is more general since it can be homogeneous with any degree of homogeneity strictly greater than one. We refer the reader to \cite{Iwaniec2002} for higher dimensional generalizations of $B_p$ and to \cite{Astala2012} for extremality results concerning this integrand.

\section{Homogeneity properties of a class of rank-one convex integrands}
\label{sec:hom}

In this section we discuss homogeneity properties of rank-one convex integrands which vanish on cones, both in two and in higher dimensions. We are interested in the family of isotropic cones of aperture $a\geq 1$,
$$\eu{C}_a\equiv \{(z,w)\in \C^2: a|z|=|w|\},$$
motivated by the fact that the Burkholder function $B_p$ vanishes on $\eu{C}_{p^*-1}$. When $a=1$ we have $\eu{C}_1=\{\det=0\}$, which of course can be defined in any dimension.

\begin{lema}\label{lema:hom}
Let $E\colon \C \times \C\to \R$ be rank-one convex  and assume that, for some $a\geq 1$, $E$ is non-positive on $\eu{C}_a$.
Define
$$h_a(t,k)=\frac 1 t
\frac{(a-k)[t(k-1)(a-1)-(a+1)(k+1)]}{(a+k)[t(k-1)(a+1)-(a-1)(k+1)]}$$  and let $A=(z,w)$ be such that $k\equiv |w|/|z|\leq 1$. Then, for $t\geq 1$,
\begin{align*}
& E(A)\leq h_a(t,k) E(t A).
\end{align*}
\end{lema}

\begin{proof}
Let us write, for real numbers $x,y \in \R$, $(x,y)\equiv (x z/|z|, y w/|w|)\in \C\times \C$, so $A=(|z|,k|z|)$. We fix $t>1$ since when $t=1$ there is nothing to prove. Let us define the auxiliary points
\begin{gather*}
A_1 = \frac {|z|} {2} \left(1+k+t- kt , 1+k - t + k t\right),\\
B_1=\frac{|z|(1+k)}{a+1}\left(1,a\right),\hs 
B_2 = \frac{t|z|(1-k)}{a+1}\left(1,-a\right),
\end{gather*}
see Figure \ref{fig:hom}. Simple calculations show that
\begin{align*}
&A=\lambda_1 A_1+ (1-\lambda_1) B_1, &\lambda_1 = \frac{2(a -  k)}{(a-1) (k+1) -t(a+1) (k-1) }\\
&A_1=\lambda_2 (tA)+ (1-\lambda_2) B_2, &\lambda_2 = \frac{1+k + t k -t + a (1 + k + t - k t)}{2 (a + k) t}
\end{align*}
and it is easy to verify that $B_1-A_1$ and $B_2-tA$ are rank-one directions. One also needs to verify that $\lambda_1, \lambda_2 \in [0,1]$, which is a lengthy but elementary calculation using the fact that $0\leq k\leq 1\leq a$.
\begin{figure}
  \centering
    \includegraphics[width=0.5\linewidth]{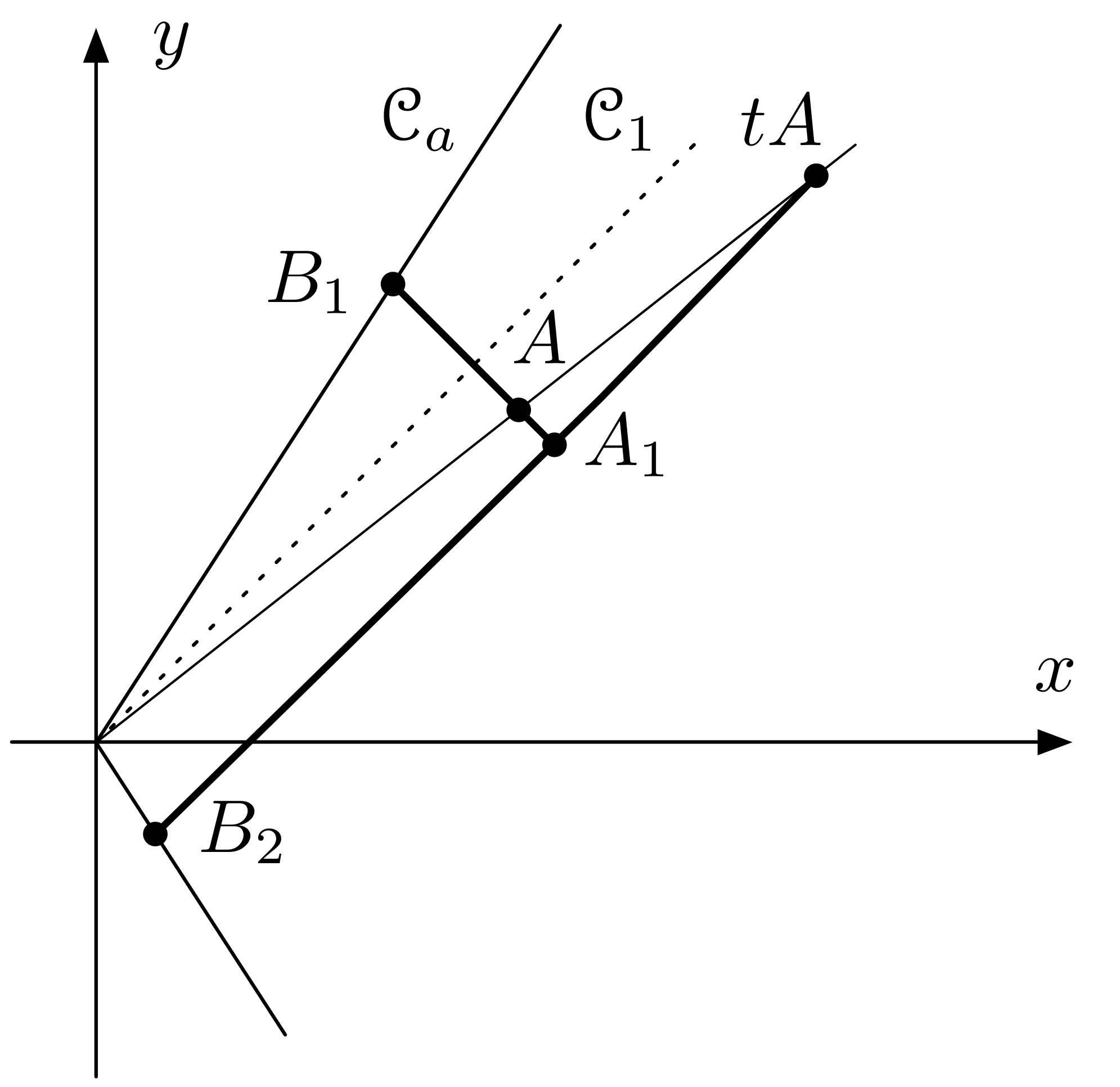}    
\caption{Bold lines are rank-one segments.}
\label{fig:hom}
\end{figure}

Observe that $B_1, B_2\in \eu{C}_a$ and so $E(B_1)=E(B_2)\leq 0$. Therefore, from rank-one convexity, we have
$$E(A) \leq (1-\lambda_1 )E(B_1)+\lambda_1(1-\lambda_2) E(B_2) + \lambda_1 \lambda_2 E(t A)\leq  \lambda_1 \lambda_2 E(tA)$$
and a simple calculation shows that $\lambda_1 \lambda_2 = h_a(t,k)$.
\end{proof}

We now specialise the lemma to two important situations, in which one can say more. Let us first assume that $k=0$.

\begin{prop}\label{prop:hombound}
Let $E\colon \C \times \C \to \R$ be rank-one convex, positively $p$-homogeneous for some $p\geq 1$ and not identically zero. If there is some $a\geq 1$ such that $E=0$ on $\eu{C}_a$ then $p\geq \frac 1 a + 1$.
\end{prop}

We note that this inequality is sharp: indeed, the zero set of the
Burkholder function $B_p$ is $\eu{C}_{p^*-1}$ and so, when $1<p< 2$, we have 
$$a= p^*-1 = \frac{1}{p-1} \Leftrightarrow p = \frac 1 a + 1.$$
Thus we can reinterpret this proposition as saying that, for $1<p< 2$, the Burkholder function has the least possible order of homogeneity of rank-one convex integrands which vanish on $\eu{C}_{p^*-1}$.

\begin{proof}
In this proof we assume that $a>1$, since the case $a=1$ follows from Lemma \ref{lema:homogeneity} below.
We  claim that there is some $z\in \C$ such that $E(z,0)> 0$. Once this is shown, the conclusion follows easily: take $k=0$ in Lemma \ref{lema:hom} 
to find that
$E(A)\leq F_a(t) E(A)$ where
$$F_a(t)=t^{p-1}\frac{a t + a - t + 1}{a t + a + t -1}$$
and $A=(z,0)$.
Since $E(A)>0$, 
we must have $F_a(t)\geq 1$ for all $t\geq 1$. Moreover, $F_a(1)=1$ and an elementary computation reveals that
$$\left.\frac{\d F_a}{\d t}\right|_{t=1}=p-1-\frac 1 a$$
which is non-negative precisely when $p\geq  \frac 1 a + 1$.

To finish the proof it suffices to prove the claim. 
Take an arbitrary $z\in S^1$ and take any rank-one line segment
starting at $(z,0)$ and having the other end-point in $\eu{C}_a$; such a line must intersect $\eu{C}_1$, since $a>1$, say at $P_z$. Note that $E(P_z)\geq 0$, since the function $t\mapsto E(t P_z)= t^p E(P_z)$ is convex. We conclude that $E(z,0)\geq 0$ with equality if and only if $E(P_z)=0$, in which case $E$ is identically zero along the entire rank-one line segment. 

To prove the claim we want to show that if $E(z,0)=0$ for all $z \in \C$ then $E$ is identically zero, so let us make this assumption. Then, from the previous discussion, we see that $E$ is identically zero in the ``outside'' of $\eu{C}_a$, i.e. in 
$$\eu{C}_a^+\equiv \{(z,w)\in \C\times \C: a |z|>|w|\}.$$
Moreover, given any point $P$ in the interior of $\eu{C}_a$, there is a rank-one line segment through $P$ with both endpoints, say $P_1, P_2$, in $\eu{C}_a^+$; this is the case because we assume $a>1$. But $E$ is zero in a neighbourhood of $P_i$ and since it is convex along the rank-one line segment $[P_1,P_2]$  we conclude that it is also zero at $P$.
\end{proof}

We remark that, in one dimension, the only homogeneous extreme convex integrands are linear (c.f. Proposition \ref{prop:extremes})
while, from the results of Section \ref{sec:maintheorem}, for $n>1$ there are extremal rank-one convex integrands in $\R^{n\times n}$ which are positively $k$-homogeneous for any $k\in \{1,\dots, n\}$. It would be interesting to know whether there are extremal homogeneous integrands with other degrees of homogeneity, or whether there is an upper bound for the order of homogeneity of such integrands.

If we set $a=1$ in Lemma \ref{lema:hom}, so $\eu{C}_1=\{\det=0\}$, we see that $h_1(t,k)=t^{-2}$ and we find the estimate $t^2 E(A)\leq E(tA)$ for $t\geq 1$. In fact, this holds in any dimension, and the proof is a simple variant of the proof of Lemma \ref{lema:hom}.

\begin{lema}\label{lema:homogeneity}
Let $E\colon \R^{n\times n} \to \R$ be a rank-one convex integrand which is
non-positive on the cone $\{\det=0\}$. Then for all $A\in \R^{n\times n}$, $E$ satisfies
\begin{align*}
&t^n E(A) \leq E(tA) \hs\textup{ if } 1\leq t.\\
& t^n E(A) \geq E(tA) \hs\textup{ if } 0<t\leq 1.
\end{align*}
\end{lema}

\begin{proof}
Let us begin by observing that the second inequality follows from the first. Indeed, given a matrix $A\in \R^{n\times n}$ and
$0<t<1$, let $B\equiv t A$ and apply the first inequality to $B$ to get
$$E(tA)=E(B)= t^n \frac{1}{t^n} E(B) \leq t^n E\left(\frac 1 t
  B\right) = t^n E(A);$$
this can be done since $1<\frac 1 t $. Hence we shall prove only
the first inequality. 

We begin by proving the statement in the case where $A$ is diagonal,
so there are real numbers $\sigma_j$ such that
$A=\textup{diag}(\sigma_1, \dots, \sigma_n)$. 
  Let $A_0\equiv A$ and define, for $1\leq j \leq n$,
  \begin{align*}
    & A_j = \textup{diag} (t \sigma_1, \dots, t\sigma_j, \sigma_{j+1}, \dots,
      \sigma_n)    & B_j = \textup{diag}(t \sigma_1, \dots, t \sigma_{j-1},0, \sigma_{j+1}, \dots,
      \sigma_n).
    \end{align*}
Hence, for any $1\leq j \leq n$,
\begin{equation}
A_{j-1} = \frac 1 t A_j + \frac{t-1}{t} B_j .\label{eq:splitting}
\end{equation}
This is a splitting of $A_{j-1}$ since we are assuming that $t> 1$ and also
$$A_j-B_j=\textup{diag} (0,\dots, 0, t\sigma_j, 0,
\dots, 0),$$
which is a rank-one matrix.
Iterating (\ref{eq:splitting}) we find
\begin{equation}
A_0 = \frac{1}{t^n} A_n + \sum_{j=1}^n \frac{t-1}{t^j} B_j\label{eq:prelaminate}
\end{equation}
and by construction this is a prelaminate. Thus rank-one convexity of
$E$ yields
\begin{equation}
E(A)=E(A_0)\leq \frac{1}{t^n} E(A_n) + \sum_{j=1}^n \frac{t-1}{t^j}
E(B_j) \leq \frac{1}{t^n} E(tA)\label{eq:7}
\end{equation}
since $\det(B_j)=0$ for all $1\leq j\leq n$, hence $E(B_j)\leq 0$, and
also $A_n= tA$ by definition.

In the case where $A$ is not diagonal, we consider the singular value
decomposition of Theorem \ref{teo:singval}, i.e.
$A=Q\Sigma R$ where $Q,R\in \textup{SO}(n)$ and
$\Sigma=\textup{diag}(\sigma_1, \dots, \sigma_n)$. We see that
(\ref{eq:prelaminate}) can be rewritten as
$$\Sigma  = \frac{1}{t^n} (t\Sigma) + \sum_{j=1}^n \frac{t-1}{t^j} B_j$$
and so, multiplying this by $Q$ and $R$, we get
$$A=Q\Sigma R = \frac{1}{t^n} (t A) + \sum_{j=1}^n \frac{t-1}{t^j}
QB_j R.$$
We still have $\det(Q B_j R)=\det(B_j)=0$ and hence to finish the
proof it suffices to show that this decomposition of $A$ is still a prelaminate.
For this,
we use the following elementary
fact:
\begin{equation*}
\textup{for all } A\in \R^{n\times n} \textup{ and } M\in \textup{GL}(n),  \textup{
   rank}(AM)=\textup{rank}(MA)=\textup{rank}(A). 
\end{equation*}
The splittings used to obtain this prelaminate are rotated versions of
(\ref{eq:splitting}), i.e.
$$QA_{j-1}R = \frac 1 t QA_jR + \frac{t-1}{t} QB_jR$$
and this is still a legitimate splitting since
\begin{equation*}
\textup{rank}(QB_j R-
QA_j R) = \textup{rank}(B_j-A_j)=1.\qedhere
\end{equation*}
\end{proof}

As a simple consequence of the lemma, we  find a rigidity result for decompositions of positively $n$-homogeneous integrands.

\begin{prop}\label{prop:superhom}
Let $E_1, E_2\colon\R^{n\times n}\to \R$ be rank-one convex integrands which are non-positive on $\{\det=0\}$ and assume there is some positively $n$-homogeneous integrand $F$ such that $F=E_1+E_2$. Then each $E_i$ is positively
$n$-homogeneous.
\end{prop}

\begin{proof}
Define the ``homogenized'' integrands $E_i^h\colon\R^{n\times n}\to
\R$ by
$$E_i^h(A)\equiv 
\begin{cases}
|A|^n E_i\left(\frac{A}{|A|}
\right) & A\neq 0\\
0 & A=0
\end{cases}
$$
so that the lemma yields
\begin{equation*}
  \begin{split}
   & \textup{if }|A|<1 \textup{ then } E_i(A)\leq E_i^h(A),\\
&\textup{if }|A|>1 \textup{ then } E_i(A)\geq E_i^h(A).
  \end{split}
\end{equation*}
Our claim is that $E_i=E_i^h$.
Since $F= E_1+E_2$ it follows that
$$F\leq E_i^h +  E_2^h \hs\textup{in }
U\equiv\{A\in \R^{n\times n}: |A|\leq 1\}.$$
and we have equality on the sphere $\{A\in \R^{n\times n}: |A|= 1\}$,
where $E_i=E_i^h$. As both sides
of the inequality are positively 
$n$-homogeneous they must be equal in the whole set $U$ and so
$E_i=E_i^h$ in $U$. An identical argument establishes equality in the
complement of $U$.
\end{proof}

\begin{remark}
  The proofs of Lemma \ref{lema:homogeneity} and Proposition \ref{prop:superhom} are fairly robust. In
  particular, a similar statement holds if the integrands $E_i$ are
  defined in $\R^{n\times n}_{\textup{sym}}$ instead of $\R^{n\times
    n}$. Indeed, $\R^{n\times n}_{\textup{sym}}$ is the set of
 (real) matrices that can be diagonalized by rotations. Thus, the prelaminate
 built in the proof of Lemma \ref{lema:homogeneity} has support in
 $\R^{n\times n}_{\textup{sym}}$ if $A$ is symmetric: for the
 nondiagonal case, one can take $R=Q^{-1}$.
\label{rem:symhom}
\end{remark}

Returning to the case $n=2$, it would be pleasant to have a result
analogous to Proposition \ref{prop:superhom} for integrands vanishing
on cones of aperture greater than one. However, this is not possible,
since it would yield the extremality of the Burkholder function in the
class of isotropic rank-one convex integrands. In order to see that
this cannot be the case, we recall that \textsc{\v Sverák}  introduced in \cite{Sverak1990} the rank-one convex function 
$$L\left(z,w\right)\equiv \begin{cases}
|z|^2-|w|^2 & |z|+|w|\leq 1\\
2 |z|-1 & |z|+|w|\geq 1
\end{cases},$$
which is related to the Burkholder function $B_p$, when $1<p<2$, by
$$\int_0^\infty t^{p-1}L\left(\frac z t,\frac w t\right) \d t = \frac{2}{p(2-p)} B_p(z,w);$$
see \cite{Baernstein1997a}. In particular, this shows that one cannot drop the homogeneity assumption from the results of \cite{Astala2012}.

\section{Proof of extremality for truncated minors}
\label{sec:maintheorem}

This section is dedicated to the proof of Theorem \ref{teo:intro}.
Although  truncated minors are not linear along rank-one lines, they
are piecewise linear along such lines. For this reason, it will be
useful to have at our disposal the classification of rank-one affine
integrands, which is due to \textsc{Ball} \cite{Ball1977}
in dimensions three or lower, \textsc{Dacorogna} \cite{Dacorogna2007} in higher
dimensions and also \textsc{Ball-Currie-Olver} \cite{Ball1981} in the
case of higher order quasiconvexity. Given an open set $\mc{O}\subset \R^{n\times n}$ and an integrand $E\colon \mc{O}\to \R$ we say that $E$ is \emph{rank-one affine} if both $E$ and $-E$ are rank-one convex; such integrands are also often called \emph{Null Lagrangians} or \emph{quasiaffine}.

\begin{teo}\label{teo:nulllagrangians}
Let $\mc{O}\subset \R^{n\times n}$ be a connected  open set and consider a rank-one affine integrand $E\colon \mc{O}\to \R$. Then $E(A)$ is an affine combination of the minors of $A$.

More precisely, let $\mathbf{M}
(A)$ be the matrix consisting of the minors of $A$ and let $\tau(n)\equiv (2n)!/(n!)^2$ be its length. There is a constant $c\in \R$ and a vector $v\in \R^{\tau(n)}$ such that
$$E(A)=c+v \cdot \mathbf{M}(A) \hs \textup{ for all } A \in \mc{O}.$$
\end{teo}

This theorem is essentially a particular case of \cite[Theorem 4.1]{Ball1981}, the only difference being that in this paper the authors deal only with integrands defined on the whole space. We briefly sketch how to adapt their proof to our case. The first result needed is the following:

\begin{lema}\label{lema:rank-oneaffine}
Let $\mc{O}\subset \R^{n\times n}$ be open. A smooth integrand $E\colon \mc{O}\to \R$ is  rank-one affine  if and only if, for any $k\geq 2$,
$$D^k E(A)[v_1\otimes w_1, \dots, v_k\otimes w_k] =0$$
for all $A\in \mc{O}$ and all $v_i,w_i \in \R^n$ with $w_1,\dots, w_k$ linearly dependent.

In particular, when $\mc{O}$ is connected, any continuous rank-one affine integrand $E$ is a polynomial of degree at most $n$.
\end{lema}

We remark that our proof is very similar to the one in \cite[Theorem 4.1]{Ball1977}.

\begin{proof}
We recall that a  smooth integrand $E$ is rank-one affine if and only if
\begin{equation}
\label{eq:rankoneaffine}
D^2 E(A)[v\otimes w, v \otimes w]=0 \textup{ for all } v, w\in \R^n
\end{equation}
and so clearly one of the directions of the lemma holds. Hence let us assume that $E$ is rank-one affine
and fix some point $A\in \mc{O}$. Define the $2k$-tensor
$T\colon (\R^n)^{2k}\to \R$ by
$$T[v_1,\dots, v_k, w_1, \dots, w_k]\equiv D^kE(A)[v_1\otimes w_1, \dots, v_k\otimes w_k].$$
 Moreover,
 since $E$ is rank-one affine, $T$ is alternating. This follows from the following claim:  if $w_j=w_l$ for some $j\neq l$, then
 $$T[v_1,\dots, v_k, w_1, \dots, w_k]=0.$$ 
 To see why this claim is true, we note that it certainly holds when $k=2$, since (\ref{eq:rankoneaffine}) implies that
\begin{align*}
0&=D^2 E[(v_1+v_2)\otimes w, (v_1+v_2)\otimes w]\\
&=D^2 E(A)[v_1\otimes w, v_2 \otimes w]+
D^2 E(A)[v_2\otimes w, v_1 \otimes w]
\end{align*}
since 
$D^2 E(A)[v_1\otimes w, v_1 \otimes w]=0$ and the same with $v_2$ in the place of $v_1$. 
 For a general $k\geq 2$, we use implicit summation to see that
 \begin{align*}
D^kE(A)[v_1\otimes w_1, \dots, v_k\otimes w_k]&=
\frac{\p^k E(A)}{\p A_{i_1}^{\a_1}\dots   \p A_{i_k}^{\a_k} } v_1^{i_1}\dots v_k^{i_k} w_1^{\a_1} \dots w_k^{\a_k}
\\
&\hspace{-2cm}= 
\frac{\p^{k-2}}{\p A_{i_1}^{\a_1}\dots 
\widehat{\p A_{i_j}^{\a_j}}\dots
\widehat{\p A_{i_l}^{\a_l}}\dots
  \p A_{i_k}^{\a_k} }
\left[
\frac{\p^2 E(A)}{\p A_{i_j}^{\a_j} \p A_{i_l}^{\a_l}}v_j^{i_j} v_l^{i_l}
w_j^{\a_j} w_l^{\a_l}
\right]\times \\
&\hspace{-2cm} \times v_1^{i_1}\dots \widehat{v_j^{i_j}} \dots \widehat{v_l^{i_l}}\dots v_k^{i_k}w_1^{\a_1} \dots\widehat{w_j^{\a_j}}
\dots \widehat{w_l^{\a_l}} \dots w_k^{\a_k}
 \end{align*}
 where
$\widehat{\cdot}$ represents an omitted term. Now we can apply  the $k=2$ case to the term in square brackets to see that
  $$T[v_1, \dots, v_k, w_1, \dots, w_k]=0$$ as wished.

 To prove the lemma under the assumption that $E$ is smooth, let us
take  $w_1, \dots, w_k$  linearly dependent, so we can suppose for simplicity that $w_k=w_1+\dots+ w_{k-1}$. Then
$$T[v_1, \dots, v_k, w_1, \dots, w_{k-1}, w_1+\dots + w_{k-1}]=0$$
since $T$ is linear and alternating. The last statement of the lemma follows by observing that the first part implies that $D^{n+1}E(A)=0$ for all $A\in \mc{O}$.

When $E$ is merely continuous, let $\rho$ be the standard mollifier and let $\rho_\e(A)=\e^{-n^2}\rho(A/\e)$ for $\e>0$. Fix $A\in \mc{O}$ and find an $\e>0$ such that $\textup{dist}(A,\p \mc{O})>\e$. Then $E_\e\equiv \rho_\e \ast E$ is smooth and rank-one affine and hence
$$D^kE_\e(A)[v_1\otimes w_1, \dots, v_k\otimes w_k] =0$$
whenever $w_1, \dots, w_k$ are linearly dependent. Since $D^k E_\e $ converges to $D^k E$ locally uniformly, the conclusion of the lemma follows.
\end{proof}

Using the lemma, we see that in order to prove the theorem it suffices to consider rank-one affine integrands which are homogeneous polynomials, so let us take such an integrand $E$ which is a homogeneous polynomial of some degree $k$. Given any $A\in \mc{O}$, the total derivative $D^k E(A)$ is a symmetric $k$-linear function $D^k E(A)\colon (\R^{n\times n})^k\to \R$; we remark that this operator has as domain the whole matrix space and not just a subset of it. There is an isomorphism between the space of $k$-homogeneous rank-one affine integrands and the space of symmetric  $k$-linear functions $(\R^{n\times n})^k\to \R$ and the proof in \cite{Ball1981} is unchanged in our case.

We recall that a generic symmetric rank-one matrix is of the form $c v\otimes v$ for some $v\in \R^n$ with $|v|=1$ and some $c \in \R$. Hence, we have the following analogue of Lemma \ref{lema:rank-oneaffine}:

\begin{lema}
Let $\mc{O}\subset \R^{n\times n}_\textup{sym}$ be open. A smooth integrand $E\colon \mc{O}\to \R$ is  rank-one affine  if and only if, for any $k\geq 2$,
$$D^k E(A)[v_1\otimes v_1, \dots, v_k\otimes v_k] =0$$
for all $A\in \mc{O}$ and all $v_i,w_i \in \R^n$ with $v_1,\dots, v_k$ linearly dependent.
\end{lema}

From this we deduce, by the same arguments as in the situation above, the following result:

\begin{teo}\label{teo:symnulllagrangians}
Let $\mc{O}\subset \R^{n\times n}_\textup{sym}$ be a connected  open set and consider a rank-one affine integrand $E\colon \mc{O}\to \R$. Then there is a constant $c\in \R$ and a vector $v\in \R^{\tau(n)}$ such that
$$E(A)=c+v \cdot \mathbf{M}(A) \hs \textup{ for all } A \in \mc{O}.$$
\end{teo}

In order to apply Theorems \ref{teo:nulllagrangians}
and \ref{teo:symnulllagrangians} to the integrands we are interested in, we need to know that its supports are connected (in general, it is clear that any integrand with disconnected support cannot be extremal). Given a minor $M$, let 
$\mc{O}_M\equiv \{M>0\}$. Moreover, each $F_k$ has support in the set
$$\mc{O}_k \equiv \{A\in \R^{n\times n}_\textup{sym}: A \textup{ has index } k \textup{ and is invertible}\}.$$

\begin{lema}\label{lema:connectedcomps}
For any minor $M$ the set $\mc{O}_M$ is connected.
Moreover,
for $k=0,\dots, n$, the sets $\mc{O}_k$ are connected.
\end{lema}

\begin{proof}
For the first part, let $M$ be an $s\times s$ minor. Let us make the identification
$$\R^{n\times n}\cong \R^{s\times s} \times 
\R^{(n-s)\times (n-s)}$$
so that $M(A)=\det(P_s(A))$, where $P_s$ is the projection of $\R^{n\times n}$ onto $\R^{s\times s}$. Hence we see that, under this identification,
$$\mc{O}_M= \{A\in \R^{s\times s}: \det(A)>0\}\times 
\R^{(n-s)\times (n-s)}.$$
Since both spaces are connected and the product of connected spaces is connected, $\mc{O}_M$ is connected as well.

For the second part, note that the set $\mc{O}_k$ is the set of matrices $A$ for which there is some $Q\in \textup{SO}(n)$ and some
 diagonal matrix $\Lambda=\textup{diag}(a_1, \dots, a_k, b_1, \dots, b_{n-k})$, where $a_i<0$ and $b_j>0$,
 such that
$QAQ^T=\Lambda$. Clearly the set of $\Lambda$'s with this form can be connected to $\textup{diag}(-I_k, I_{n-k})$ by a path in $\mc{O}_k$; here $I_l$ is an $l\times l$ identity matrix. Hence it suffices to prove that there is a continuous path in $\mc{O}_k$ connecting $A=Q\Lambda Q^T$ to $\Lambda$. Such a path is given by $A(t)=Q(t)A Q(t)^T$, where $Q\colon [0,1]\to \textup{SO}(n)$ is a continuous path with $Q(0)=I, Q(1)=Q$.
\end{proof}

We are finally ready to prove the extremality of truncated minors and of \textsc{\v Sverák}'s integrands.

\begin{proof}[Proof of Theorem \ref{teo:intro}]
Let $M$ be a minor and let $E_1, E_2\colon \R^{n\times n}\to [0,\infty)$ be rank-one convex integrands such that $M^+=E_1+E_2$. 
For concreteness, let us say
$$
M\begin{bmatrix}
a_{11}& \dots &a_{1n}\\
\vdots & \ddots & \vdots\\
a_{n1} & \dots & a_{nn}
\end{bmatrix}
= 
\det\begin{bmatrix}
a_{i_1 j_1}& \dots &a_{i_1 j_k}\\
\vdots & \ddots & \vdots\\
a_{i_k j_1} & \dots & a_{i_k j_k}
\end{bmatrix}.
$$
Each $E_i$ is zero outside $\mc{O}_M$ and, in this set, it is rank-one affine, so by Theorem \ref{teo:nulllagrangians} there are constants $c_i$ and $v_i \in \R^{\tau(n)}$ such that
$$E_i(A)=c_i+v_i\cdot \mathbf{M}(A) \hs \textup{for all } A\in \mc{O}_M$$
and in fact, by continuity, this holds in the entire set $\overline{\mc{O}_M}=\{M\geq 0\}$.

Clearly we must have $c_i=0$. Let us write,  for some vectors $v_i^j\in \R^{{{n}\choose{j}}\times{{n}\choose{j}} }$,
$$E_i(A) = \sum_{j=1}^n v_i^j \cdot \mathbf{M}_j(A)\hs \tp{ in } \overline{\mc{O}_M},$$
where $\mathbf{M}_j(A)$ is the matrix of the $j$-th order minors of $A$ (this is denoted by $\tp{adj}_j(A)$ in \cite{Dacorogna2007}). 

We observe that, given some $s$ and some minor $M'$ of order $s$, there is a matrix $A$ such that $M'$ is the only minor of order $s$ that does not vanish at $A$.
Indeed, if
$$M'\begin{bmatrix}
a_{11}& \dots &a_{1n}\\
\vdots & \ddots & \vdots\\
a_{n1} & \dots & a_{nn}
\end{bmatrix}=
\det\begin{bmatrix}
a_{i'_1 j'_1}& \dots &a_{i'_1 j'_{s}}\\
\vdots & \ddots & \vdots\\
a_{i'_{s} j'_1} & \dots & a_{i_{s} j'_{s}}
\end{bmatrix}$$
then we can take a 
matrix $A$ whose only non-zero entries are the entries $a_{i'_\a j'_\a}$ for $\a \in \{1,\dots, s\}$ and set these entries to one, so $M'(A)=1$. Since all other entries of $A$ are zero we see that all other minors of order $s$ vanish at $A$. Note, moreover, that $A$ has rank $s$.

The previous observation, applied with $s=k$ and $M=M'$, shows that for $A\in \overline{\mc{O}_M}$ we have $v_i^k\cdot \mathbf{M}_k(A)= \lambda_i M(A)$, where $\lambda_i \in \R$ is the entry of $v_i^k$ corresponding to $M$. We now prove that all the vectors $v_i^j, j\neq k$, are zero. 

Let $j\leq k$ be the lowest integer for which $v_i^j\neq 0$ and suppose that $j<k$.
Given any minor $M'$ of order $j$, say $M'=e_\a\cdot \mathbf{M}_j$ for some $\a$, there is an $A$ with $\tp{rank}(A)=j$ so that $M'$ is the only minor of order $j$ that does not vanish at $A$. 
Since $A$ has rank $j$ all of its $(j+1)\times (j+1)$ minors vanish and, in particular, $M(A)=0$ and hence $A\in \overline{\mc{O}_M}$. Since $j$ is the lowest integer for which $v_i^j\neq 0$ we have
$$M(A)=0=E_i(A)=v_i^j \cdot \mathbf{M}_j(A)=(v_i^j)_\a M'(A).$$
Since $\a$ was chosen arbitrarily, this is a contradiction and hence we have $j=k$.

Let $j\geq k$ be the highest integer for which $v_i^j\neq 0$ and suppose that $j>k$. Given any minor $M'$ of order $j$, say $M'=e_\a\cdot \mathbf{M}_j$ for some $\a$, there is an $A$ such that $M'$ is the only minor of order $j$ that does not vanish at $A$; moreover, by flipping the sign of the $i_1$-th row of $A$, if need be, we can assume that $A \in \overline{\mc{O}_M}$.  Since $j>k$ is the highest integer for which $v_i^j\neq 0$, by  computing 
$$t^k M(A)= M(tA)=E_1(tA)+E_2(tA)= \sum_{j=1}^n t^j (v_1^j+v_2^j)\cdot \mathbf{M}_j(A)$$
and sending $0<t\nearrow \infty$
we see that we must have $(v_1^j+v_2^j)_\a=0$ and also that the sign of $E_i(t A)$ is, for large $t$, the sign of $(v_i^j)_\a M'(A)$; hence $(v_i^j)_\a=0$ for $i=1,2$. Moreover, since $\a$ was chosen arbitrarily we have $v_i^j=0$ and  we find a contradiction; thus  $j=k$.

The two previous paragraphs show that,
in $\overline{\mc{O}_M}$,
$$E_i(A) = v_i^k \cdot \mathbf{M}_k(A),$$
and we already know that $v_i^k\cdot \mathbf{M}_k(A)= \lambda_i M(A)$, so 
 the proof that $M^+$ is extremal is complete.
The fact that $M^-$ is extremal follows from the identity $M^-=M^+(J\cdot)$, where $J=(j_{\a\beta})$, 
$j_{\a \beta}=\delta_{\a \beta}(1
-2\delta_{\a i_1})
,$ so $J$ changes the sign of the $i_1$-th row.

For the second part of the theorem take some $k\in \{0,\dots, n\}$ and assume that there are rank-one convex integrands $E_1, E_2\colon \R^{n\times n}_\tp{sym}\to [0,\infty)$ such that $F_k=E_1+E_2$. The integrand $F_k$ has support in $\mc{O}_k$, which by Lemma \ref{lema:connectedcomps} is connected, and in this set each $E_i$ is rank-one affine, so Theorem \ref{teo:symnulllagrangians} implies that there are $c_i\in \R, v_i \in \R^{\tau(n)}$ such that
$$E_i(A)=c_i+v_i\cdot \mathbf{M}(A) \hs \textup{for all } A\in \mc{O}_k.$$
By continuity this in fact holds in $\overline {\mc{O}_k}$. From Remark \ref{rem:symhom} we see that each $E_i$ has to be positively $n$-homogeneous and therefore $E_i= \a_i \det$ in $\mc{O}_k$, where $\a_i$ is the last entry of $v_i$.  Since $E_i\geq 0$ we must have, by possibly changing the sign of $\a_i$, $E_i=\a_i |\det|$ in $\mc{O}_k$. Moreover, $E_i=0$ outside $\mc{O}_k$, and so indeed $E_i=\a_i F_k$ as wished.
\end{proof}

We note that, for the second part of the theorem, it is helpful to employ the homogeneity from Remark \ref{rem:symhom}. Indeed, it is easy to see, and it follows in particular from the linear algebraic arguments in the proof of the full space case, that minors of a fixed order are linearly independent as functions on $\R^{n\times n}$. However, this is not the case if instead we think of them as functions defined over $\R^{n\times n}_\tp{sym}$, since there are non-trivial linear relations between minors. For instance, given a $4\times 4$ matrix $A=(a_{ij})$, we have that 
$$-(a_{13}a_{24}-a_{14}a_{23})+(a_{12}a_{34}-a_{14}a_{23})-(a_{12}a_{34}-a_{13}a_{24})=0.$$
This is classical phenomenon and had already been noted, for instance, in \cite[Pages 155-156]{Ball1981}.

\section{Choquet theory and Morrey's problem}
\label{sec:choquet}

In this section we shall see the implications of Choquet theory for Morrey's problem. 
Let us introduce some notation: given a number $d\in \N$, let $Q_d=[0,1]^d$, denote by $A_1, \dots, A_{2^d}$ its vertices and consider the cone
$$\mc{C}_d^{\tp{c}}\equiv \{f\colon Q_d \to [0,\infty): f \textup{ is convex}\}.$$
In a similar fashion we define the cone $\mc{C}_d^{\tp{sc}}$ of non-negative separately convex functions on $Q_d$ and when $d=n\times m$ for some $n,m\in \N$, we have the cones $\mc{C}_d^{\tp{qc}}$ and $\mc{C}_d^{\tp{rc}}$ of non-negative quasiconvex and rank-one convex integrands.
These are  closed convex cones in the locally convex vector space $\R^{Q_d}$ of real-valued functions on $Q_d$; the topology on $\R^{Q_d}$ is the product topology, i.e. the topology of pointwise convergence. In particular, for any $x\in Q_d$ the evaluation functionals $\e_x\colon f\mapsto f(x)$ are continuous. 

We claim that each of the above cones has a compact, convex base:
$$\mc{B}_d^\Box\equiv \mc{C}_d^\Box\cap \Bigg\{f\in \R^{Q_d}:\sum_{i=1}^{2^d} f(A_i)=1\Bigg\}, \hs \Box \in \{\tp{c},\tp{qc},\tp{rc},\tp{sc}\};$$
here we only take $\Box \in \{\tp{qc}, \tp{rc}\}$ if $d=n\times m$. Clearly each $\mc{B}_d^\Box$ is a closed, convex base for $\mc{C}_d^\Box$, so 
it suffices to see that $\mc{B}^{\tp{sc}}_d$ is compact. For this, note that a separately convex function on $Q_d$ attains its maximum at some $A_i$ and, since all functions $f\in\mc{B}^\tp{sc}_d$ are non-negative, we have $f\leq 1$ in $Q_d$. This shows that $\mc{B}_d^\tp{sc}\subset [0,1]^{Q_d}$, which is a compact set by Tychonoff's Theorem, and our claim is proved.

The main tool of this section is the following powerful result:

\begin{teo}[Choquet]
Let $K$ be a  metrizable, compact, convex subset of a locally convex vector space $X$. For each $f\in K$ there is a regular probability measure $\mu$ on $K$ which is supported on the set $\tp{Ext}(K)$ of extreme points of $K$ and which \emph{represents} the point $f$: for all $\varphi \in X^*$,
$$\varphi(f) = \int_{\tp{Ext}(K)} \varphi \d \mu .$$
\end{teo}

For a proof see, for instance, \cite[\S 3]{Phelps2001}.
We note that in general---and this is also the case in our situation---the representing measure is not unique.
In order to apply this theorem to $\mc{B}_d^\Box$, we need to show that this set is metrizable and this can be done by using a simple result from point-set topology; for a proof see, for instance, \cite[Lemma 10.45]{Lukes2010}.

\begin{lema}
Let $K$ be a compact Hausdorff space. Then $K$ is metrizable if and only if there is a countable family of continuous functions on $K$ which separates points.
\end{lema}

In our situation, it is easy to see that such a family exists: indeed, let $x_n\in Q_d$ be a countable set of points which is dense in $Q_d$ and consider the evaluation functionals $\e_{x_n}\colon f\mapsto f(x_n)$ on $\mc{B}_d^\Box$. These functionals are continuous on $\mc{B}_d^\Box$ and 
separate points, since all elements of $\mc{B}_d^\Box$  are continuous real-valued functions on $Q_d$. Therefore the lemma implies that $\mc{B}_d^\Box$ is metrizable, and hence Choquet's theorem yields:

\begin{prop}\label{prop:extremes}
Fix $\Box \in \{\tp{c},\tp{qc},\tp{rc},\tp{sc}\}$ and let $\nu$ be a regular probability measure in $Q$.	The measure $\nu$ satisfies 
$f(\overline \nu )\leq \langle  \nu,f \rangle$ for all $f\in \mc{C}_d^\Box$ if and only if 
$g(\overline \nu )\leq \langle  \nu,g \rangle$ for all $g\in \tp{Ext}(\mc{B}_d^\Box)$.
\end{prop}

\begin{proof}
Assume that we have Jensen's inequality for all $g \in\tp{Ext}(\mc{B}_d^\Box)$, take any $f\in \mc{B}_d^\Box$ and let $\mu$ be the measure given by Choquet's theorem. If we take $\varphi=\e_{\overline \nu}$ in the theorem, we can apply Fubini's theorem to see that
\begin{align*}
f(\overline \nu) = \e_{\overline{\nu}}(f) &= \int_{\tp{Ext}(\mc{B}_d^\Box)}\e_{\overline \nu}(g) \d\mu(g)\\
&\leq
\int_{\tp{Ext}(\mc{B}_d^\Box)}\int_Q g \d \nu \d\mu(g) \\
&= \int_Q \int_{\tp{Ext}(\mc{B}_d^\Box)} g d\mu(g) d\nu = 
\int_Q f d\nu.
\end{align*}
Since any $h\in \mc{C}_d^\Box$ can be written uniquely as $h=\lambda f$ for some $\lambda>0, f \in \mc{B}_d^\Box$, the conclusion follows.
\end{proof}

Theorem \ref{teo:choquet} follows easily from Proposition \ref{prop:extremes}. 
For the reader's convenience, we restate the theorem here:

\begin{teo}
Let $d=n^2$ and take a Radon probability measure $\nu$ supported in the interior of $Q_d$. Then $\nu$ is a laminate if and only if
$$g(\overline \nu)\leq \langle  \nu,g \rangle $$
for all integrands $g\in \tp{Ext}(\mc{B}_d^\tp{rc})$.
\end{teo}

\begin{proof}
From Pedregal's Theorem,
the measure $\nu$ is a laminate if and only if Jensen's inequality holds for any rank-one convex integrand $f\colon \R^{n\times n}\to \R$:
$$f(\overline \nu)\leq \langle \nu , f\rangle , \hs \overline \nu \equiv \int_{\R^{n\times n}} A\textup{ d} \nu(A).$$
Note that if this inequality holds for all \emph{non-negative} rank-one convex integrands then it holds for any rank-one convex integrand: given any such $f$, one can consider the new integrand
$$f_k\equiv k+\max(f,-k)$$
which is non-negative and rank-one convex, hence by hypothesis
$f_k(\overline \nu)\leq \langle \nu, f_k \rangle $. This in turn is equivalent to
$$  \max(f(\overline \nu),-k) \leq \langle \nu, \max(f,-k)\rangle.$$
Sending $k\to \infty$ we find that
$ f(\overline \nu) \leq \langle \nu, f\rangle $,
as we wished; note that $f$, being continuous, is bounded by below on $Q_d$.

Therefore, from Proposition \ref{prop:extremes}, the theorem follows once we show that any rank-one convex integrand $g\colon Q_d\to [0,\infty)$ can be extended to a rank-one convex integrand $f\colon \R^{n\times n}\to \R$ with $g=f$ in the support of $\nu$. 
This is a standard result, see \cite{Sverak1990}.
\end{proof}

We end this section with some cautionary comments concerning the previous results.
In the one dimensional case, where all the above cones coincide, the extreme points are quite easy to identify; the oldest reference we found where this problem is discussed is \cite{Blaschke1916}, but see also \cite[\S 14.1]{Lukes2010}. 
 
\begin{prop}
The set of extreme points of $\mc{B}_1$ is the set $\{\varphi_y, \psi_y: y\in [0,1]\}$, where the functions $\varphi_y, \psi_y$ are defined by
\begin{align*}
&\varphi_y \colon x \mapsto \frac{(x-y)^+}{1-y} \textup{ for } x \in [0,1], y \in [0,1), \hs \varphi_1 = 1_{\{1\}},\\
&\psi_y \colon x \mapsto \frac{(y-x)^+}{y} \textup{ for } x \in [0,1], y \in (0,1], \hs \psi_1 = 1_{\{0\}}.
\end{align*}
\end{prop}

In higher dimensions the various cones are different. In the case of convex integrands, 
the set of extreme points of $\mc{C}_d^\tp{c}$ for $d>1$ is very different from the one-dimensional case, since it is dense in this cone.

\begin{teo}
Any finite continuous convex function on a convex domain $U\subset \R^d$ can be approximated uniformly on convex compact subsets of $U$ by extremal convex functions.
\end{teo}

This result was proved by \textsc{Johansen}
in \cite{Johansen1974}
for $d=2$ and then generalized to any $d>1$ in \cite{Bronshtein1978}. In these papers the set of extremal convex functions  is not fully identified, but it is shown that there is a sufficiently large class of extremal convex functions which approximate any given convex function well: these are certain \emph{polyhedral functions}, i.e. functions of the form $f=\max_{1\leq i \leq k} a_i$ for some affine functions $a_1, \dots, a_k$.
 This disturbing situation, however, is not too unexpected given the result of \textsc{Klee} \cite{Klee1959} already mentioned in the introduction.
I do not know whether a similar statement holds for the cones $\mc{C}_{n\times m}^\tp{qc}$ and $\mc{C}_{n\times m}^\tp{rc}$.

\bigskip

\emph{Acknowledgements:} 
This work was supported by the Engineering and Physical Sciences Research Council [EP/L015811/1].
I thank my supervisor Jan
Kristensen for guidance throughout the process of obtaining these
results as well as for carefully reading and making corrections to the
drafts. I would also like to thank István Prause, Rita Teixeira da Costa and Lukas Koch for helpful discussions and comments.

{\small
\bibliographystyle{acm}


\bibliography{/scratch/Papers-SveraksConj.bib}
}

\end{document}